\documentclass[leqno,11pt, twoside]{article} 
\usepackage{amsmath, amsthm, amssymb}

\usepackage{geometry}
\usepackage{fancyhdr} 
\usepackage{titlesec}

\titleformat{\section}
  {\normalfont\scshape}{\thesection .}{1em}{}
\titleformat{\subsection}
  {\normalfont\scshape}{\thesubsection}{1em}{}

\title{An effective bound for the gonality conjecture}
\date{}
\author{\scshape J\"urgen Rathmann}

\makeatletter
\let\runauthor\@author
\let\runtitle\@title
\makeatother

\usepackage{mathrsfs}  
\usepackage[cmtip]{xypic}
\usepackage{enumitem}
\setlist{nolistsep}

\newtheorem{theorem}{Theorem}[section]
\newtheorem{proposition}[theorem]{Proposition}
\newtheorem{lemma}[theorem]{Lemma}
\newtheorem{corollary}[theorem]{Corollary}

\theoremstyle{remark}
\newtheorem{remark}[theorem]{Remark}

\renewcommand\O{\mathscr{O}}
\renewcommand\P{\mathbf{P}}
\renewcommand\lto{\longrightarrow}
\DeclareMathOperator{\gon}{gon}
\DeclareMathOperator{\Sym}{Sym}
\DeclareMathOperator{\Ker}{Ker}

\DeclareMathOperator{\Hom}{Hom}

\DeclareMathOperator{\ev}{ev}
\DeclareMathOperator{\rank}{rank}

\begin{document}
\bibliographystyle{plain}
\maketitle

\section{Introduction}

Ein and Lazarsfeld have shown that one can read off the gonality of an algebraic curve from 
its syzygies in the embedding defined by any one line bundle of sufficiently large degree. This note extends
their approach and shows that the gonality can be detected from the syzygies of an embedding by
any line bundle of degree at least $4g-3$. 
\bigskip

Let $C\subset\P H^0(L)=\P^r$ be a smooth complex projective curve of genus $g\ge 2$, embedded by a very ample 
line bundle $L$ of degree $d$.

We denote by $S=\Sym H^0(C,L)$ the homogeneous coordinate ring of $\P^r$, by
$R=R(L)=\oplus_mH^0(C,mL)$
the graded $S$-module associated to $L$, and by $E_\bullet=E_\bullet(L)$ the minimal graded free resolution 
of $R$ over $S$:
\begin{equation*}
0\lto E_{r-1}\lto\dotsc\lto E_2\lto E_1\lto E_0\lto R\lto 0.
\end{equation*}
Let $K_{p,q}(C;L)$ be the vector space of minimal generators of $E_p$ in degree $p+q$, so that
\begin{equation*}
E_p=\oplus_qK_{p,q}(C;L)\otimes S(-p-q).
\end{equation*}

If $L$ is normally generated, then $E_0=S$, and the remainder of the complex $E_\bullet$ is a minimal
graded resolution of the homogeneous ideal $I_{C/\P^r}$ of $C$ in $\P^r$.

If $L$ is nonspecial then $K_{p,q}(C;L)$ vanishes for $q\ge 3$, hence for $\deg(L)\ge 2g+1$ the resolution
essentially consists of two strands $K_{p,1}(C;L)$ and $K_{p,2}(C;L)$.

For a discussion of these strands, see \cite{GL} and \cite[ch.\ 8B]{Eis}.

The gonality conjecture of Green and Lazarsfeld \cite{GL} concerns the ``quadratic strand'' $K_{p,1}(C;L)$ for
line bundles of large degree. If $C$ has a basepointfree pencil of degree $k$, then the linear spans of the fibers 
of this pencil form a $k$-dimensional rational normal scroll $S$. The minimal graded resolution of its ideal $I_{S/\P^r}$
has a linear strand of length $r-k$ which embeds into the resolution of $I_{C/\P^r}$.

Green and Lazarsfeld conjectured in 1986 and Ein and Lazarsfeld proved in 2014 that such pencils 
determine the range of values for $p$ where $K_{p,1}(C;L)$ does not vanish, if the degree of $L$ is large.

We show that this already holds for $\deg L\ge 4g-3$.

\begin{theorem}\label{thm11}
If $H^1(L\otimes K^{-1})=0$\textup{,} then
\begin{equation*}
K_{p,1}(C;L)\ne 0 \iff 1\le p\le r-\gon(C).
\end{equation*}
In particular\textup{,} the equivalence holds if $\deg(L)\ge 4g-3$.
\end{theorem}

Although our result represents a significant step forward, we do not expect the bound $4g-3$ to be optimal.
Aprodu and Voisin \cite{Apr1}, \cite{AprVoisin} found that $\deg(L)\ge 3g$ 
suffices for a general curve of each gonality.
Green \cite{Green1} had earlier demonstrated the conclusion for $p=r-1$ resp.\ $p=r-2$ under even lower 
bounds of $\deg(L)\ge 2g+1$ resp.\ $\deg(L)\ge 2g+2$ and $r\ne 5$ (the exceptions for $r=5$ are 
smooth curves of genus 3 and degree $8=2g+2$ on the Veronese surface in $\P^5$).  

\bigskip
While it is not hard to exhibit a non-vanishing syzygy from a pencil on $C$, it is not
known how to construct a pencil out of a non-vanishing syzygy, except for the cases $p=r-1$ and $p=r-2$ covered
by Green. Ein und Lazarsfeld's proof works by contradiction. They assume
that $C$ does not have a pencil of low degree, and prove the vanishing of the related syzygy module.

Following their arguments, \ref{thm11} can be derived from a more general result on the syzygies associated to
a line bundle $B$. Letting $R=R(B;L)=\oplus_m H^0(C,B+mL)$, its minimal graded free resolution $E_\bullet(B;L)$ 
over $S$ gives rise to Koszul cohomology groups $K_{p,q}(C,B;L)$.

Recall that $B$ is called $p$-very ample if every effective divisior $\xi$ of degree $(p+1)$ on $C$ imposes 
independent conditions on the sections of $B$; e.g., the canonical divisor $K_C$ is $p$-very ample
if and only if $C$ does not have a pencil of degree $<p+2$.

Our main result is a strengthening of one direction of \cite[Thm.\ B]{EinLaz}.

\begin{theorem}\label{thm12}
Let $B$ be a $p$-very ample line bundle. If $H^1L=0=H^1(B^{-1}\otimes L)$\textup{,} then $K_{p,1}(C,B;L)=0$.
\end{theorem}

As explained by Ein and Lazarsfeld, Theorem \ref{thm11} follows from setting $B=K_C$ in 
\ref{thm12}, using the duality
between $K_{p,q}(C,B;L)$ and $K_{r-1-p,2-q}(C,K_C\otimes B^{-1};L)$.

\bigskip
For the proof of \ref{thm12}, we use the representation of $K_{p,1}(C,B;L)$ 
as cokernel of a map of global sections of sheaves 
on the symmetric product $C_{p+1}$ \cite{EinLaz}. 

Our analysis starts by transferring the question from the symmetric to the cartesian product.
Denoting by $pr_{p+2}\colon C^{p+2}\lto C$ resp.\ $\pi_{p+2}\colon C^{p+2}\lto C^{p+1}$ the projection 
maps on the $(p+2)$-nd resp. first $p+1$ components, by $\Delta_{i,j}$ the diagonals $x_i=x_j$ in $C^{p+2}$, we set
\begin{equation*}
M_B=M_{p+1, B}=\pi_{p+2,*}\big(pr_{p+2}^* (B)\otimes\O(-\sum_{i=1}^{p+1}\Delta_{i,p+2})\big),
\end{equation*}
and we show that if $B$ is $p$-very ample, the vanishing of $K_{p,1}(C,B;L)$ follows from the vanishing of
\begin{equation*}
H^1\big(C^{p+1},M_B\otimes(\otimes_{i=1}^{p+1}pr_i^*(L))\otimes\O(-\sum_{1\le i<j\le p+1}\Delta_{i,j})\big).
\end{equation*}

There are two standard approaches for proving such vanishing theorems:
\begin{itemize}
\item[(i)]
using a filtration of $M_B$, and
\item[(ii)]
using a resolution of $M_B$ derived from 
a representation of $M_B$ as a kernel bundle.
\end{itemize}

If we have no information on $B$ except that it is $p$-very ample, only method (ii) is at our disposal. 
Theorem \ref{thm12} represents this case (proved in \ref{thm31}), and the required representation of $M_B$ as a 
kernel bundle is given by
\begin{equation*}
0\lto M_{p+1,B}\lto \pi_{p+1}^*M_{p,B}\lto pr_{p+1}^*(B)\otimes\O\big(-\sum_{i=1}^p\Delta_{i,p+1}\big)\lto 0
\end{equation*}
which allows us to proceed by induction.

If there is a point $x\in C$ such that $B(-x)$ is again $p$-very ample, then we can start to filter $M_B$ as
\begin{equation*}
0\lto M_{B(-x)}\lto M_B\lto \O(-x,\dotsc,-x)\lto 0.
\end{equation*}
Improved bounds can be obtained if the degree of $B$ is large compared to $g$ 
(see \ref{prop32} and \ref{prop36}); for $B=K_C$ (the canonical bundle) this approach 
leads to an improvement if and only if $K$ is actually $(p+1)$-very ample (see \ref{prop310}).

What happens if $K$ is not $(p+1)$-very ample? For $p=0$, i.e., hyperelliptic $C$, the bundle
$M_{1,B}$ is the the direct sum of $g-1$ copies of $(g^1_2)^{-1}$, and $K_{0,1}(C,K;L)$ vanishes
for $\deg L\ge 2g+1$. For $p=1$ we know from Green's work that $\deg L\ge 2g+2$ 
($r\not=5$) suffices.
It seems likely that similar improvements are possible for higher $p$; however, they require
a different approach.

\bigskip
In section 2 we recall Voisin's representation of the Koszul cohomology groups, following Ein and Lazarsfeld. 
The third section contains the proofs of our main results. Section 4 discusses improvements
for $B=K$, taking the geometry of the canonical embedding into account.

\section{Review of the setup}

We recall the setup used by Ein and Lazarsfeld.

$C\subset\P H^0(L)=\P^r$ is a smooth complex projective curve of genus $g$, embedded 
by a very ample line bundle $L$ of degree $d$. Given another line bundle $B$ on $C$, let
\begin{equation*}
K_{p,q}(B;L)=K_{p,q}(C,B;L)
\end{equation*}
be the cohomology groups of the complex
\begin{multline*}
\wedge^{p+1}H^0(L)\otimes H^0(B\otimes L^{\otimes (q-1)}) \lto \wedge^pH^0(L)\otimes 
H^0(B\otimes L^{\otimes q})\\
\lto\wedge^{p-1}H^0(L)\otimes H^0(B\otimes L^{\otimes(q+1)}).
\end{multline*}

Denoting by $C_k$ the $k$-th symmetric product of $C$, the map
\begin{equation*}
\sigma_{p+1}\colon C\times C_p\lto  C_{p+1},\enspace \enspace (x,\xi)\mapsto x+\xi
\end{equation*}
realizes $C\times C_p$ as the universal family of degree $p+1$ divisors over $C_{p+1}$. 

Now let
\begin{equation*}
E_B=E_{p+1,B}=\sigma_{p+1,*}pr_1^*(B)
\end{equation*}
where $pr_i$ is the projection on the $i$-th factor. $E_B$ is a vector bundle of rank $p+1$ 
on $C_{p+1}$, and $H^0(C_{p+1},E_B)=H^0(C,B)$. We therefore have a homomorphism
\begin{equation}\label{eq1}
\ev_B=\ev_{p+1,B}\colon H^0(C,B)\otimes\O_{C_{p+1}}\lto E_B
\end{equation}
of vector bundles on $C_{p+1}$. By definition, $\ev_B$ is surjective if and only if 
$B$ is $p$-very ample. Setting
\begin{equation*}
N_L=N_{p+1,L}=\det E_L,
\end{equation*}
$\wedge^{p+1}\ev_L$ determines an isomorphism \cite{Voisin1}
\begin{equation*}
\wedge^{p+1}H^0(C,L)\lto H^0(C_{p+1},N_L).
\end{equation*}
Tensoring $\ev_B$ by $N_L$, now consider the following map of vector bundles on $C_{p+1}$:
\begin{equation}\label{eq2}
H^0(C,B)\otimes N_L\lto E_B\otimes N_L:
\end{equation}

\begin{lemma}[Voisin, Ein-Lazarsfeld {\cite[Lemma 1.1]{EinLaz}}] 
The global sections of $E_B\otimes N_L$ can be identified with the space
\begin{equation*}
Z_{p,1}(B;L)=\Ker\big(\wedge^pH^0(L)\otimes H^0(B\otimes L)\lto\wedge^{p-1}H^0(L)\otimes 
H^0(B\otimes L^{\otimes 2})\big)
\end{equation*}
of Koszul cycles. Under this identification\textup{,} the homomorphism
\begin{equation*}
H^0(C,B)\otimes H^0(C_{p+1},N_L)=H^0(C,B)\otimes\wedge^{p+1}H^0(C,L)\lto 
H^0(C_{p+1},E_B\otimes N_L)
\end{equation*}
arising from \textup{(\ref{eq2})} agrees with the Koszul differential. In particular\textup{,}
\begin{equation*}
K_{p,1}(C,B;L)=0
\end{equation*}
if and only if the bundle map \textup{(\ref{eq2})} is surjective on global sections.
\end{lemma}

Ein and Lazarsfeld complete their proof by applying general vanishing theorems
to the kernel bundle of (\ref{eq2}). They obtain a bound of $\deg(L)>(p^2+p+2)(g-1)+(p+1)\deg(B)$
as a sufficient condition for the vanishing.

\bigskip
Our contribution starts at this point. 
We begin by transferring the question from the symmetric product to the cartesian product. 

Denote by $\pi=\pi_n\colon C^n\to C^{n-1}$ the projection on the first $n-1$ components 
(forget $n$-th component) and by $r=r_n\colon C^n\to C_n$ the canonical projection. 

The map (\ref{eq2}) pulls back by $r$, and by taking global sections, we obtain a commutative diagram
\begin{equation}\label{eq5}
\begin{aligned}
\xymatrix@C=3cm{
H^0(C,B)\otimes H^0(N_L) \ar[r]^-{H^0(\ev_B\otimes N_L)} \ar[d] & H^0(E_B\otimes N_L) \ar[d] \\
H^0(C,B)\otimes H^0 \big(r^*N_L\big) \ar[r]^-{H^0(r^*(\ev_B\otimes N_L))} & 
H^0\big(r^*(E_B\otimes N_L)\big).
}
\end{aligned}
\end{equation}

\begin{lemma}\label{lm22}
If the horizontal map on the bottom of the diagram \textup{(\ref{eq5})} is surjective\textup{,} 
then the horizontal map on the top is also surjective.
\end{lemma}
\begin{proof}
Given a section $s_0$ in the top right, consider its image $s$ in the bottom right cohomology group. 
This image is invariant under the action of the symmetric group $S_{p+1}$. 
Provided that the bottom horizontal map is surjective, we can average a preimage $s'$ in the bottom
left over all translates (using the characteristic 0 assumption) in order to arrive at an $S_{p+1}$-invariant preimage $s''$.
The $S_{p+1}$-invariance then implies \cite[(5.2)]{AprNag} that $s''$ lies in the image of the (injective) vertical map 
from the top left. This provides the sought after preimage of $s_0$.
\end{proof}

\section{Vanishing results}

We continue in the setup from section 2:

$C\subset\P H^0(L)=\P^r$ is a smooth complex projective curve of genus $g$, embedded by a very ample 
line bundle $L$ of degree $d$, $B$ a $p$-very ample line bundle on $C$. We work on the cartesian product $C^{p+2}$.

The map $\ev_B$ from (\ref{eq1}) above pulls back to a map of sheaves on $C^{p+1}$
\begin{equation*}
\ev'=r^*(\ev_B)\colon H^0(C,B)\otimes\O_{C^{p+1}}\lto r^*E_B
\end{equation*}
that arises from the canonical map
\begin{equation*}
\O_{C^{p+2}}\lto \O_{\sum_{i=1}^{p+1}\Delta_{i,p+2}}
\end{equation*}
by tensoring with $pr_{p+2}^*(B)$ and applying $\pi_{p+2,*}$. Further we have
\begin{equation*}
r^*N_L=\otimes_{i=1}^{p+1}pr_i^*(L)\otimes\O\big(-\sum_{1\le i<j\le p+1}\Delta_{i,j}\big)
\enskip\enskip\text{\cite[proof of (5.2)]{AprNag}.}
\end{equation*}

Denoting the vector bundle
\begin{equation*}
M_B=M_{p+1, B}=\pi_{p+2,*}\big(pr_{p+2}^* (B)\otimes\O(-\sum_{i=1}^{p+1}\Delta_{i,p+2})\big),
\end{equation*}
the surjectivity of the bottom horizontal map of diagram (\ref{eq5}) from the previous section would follow from
\begin{equation}\label{eq6}
H^1\big(C^{p+1},M_B\otimes(\otimes_{i=1}^{p+1}pr_i^*(L))\otimes\O(-\sum_{1\le i<j\le p+1}\Delta_{i,j})\big)=0
\end{equation}

We now show by induction on $p$ more generally
\begin{theorem}\label{thm31}
If $L$ is non-special and $H^1(C,B^{-1}\otimes L)=0$\textup{,} then we have
\begin{equation}\label{eq7}
H^k\big(C^{p+1},\wedge^m M_B\otimes pr_1^*(L)\otimes\dotsm \otimes pr_{p+1}^*(L)
\otimes\O(-\sum_{1\le i<j\le p+1}\Delta_{i,j})\big)=0
\end{equation}
for all $k>0$\textup{,} $m>0$.
\end{theorem}
\begin{proof}
First consider $p=0$: In this case we know that $B$ is globally generated and we have an exact sequence
\begin{equation*}
0\lto M_{1,B} \lto H^0B\otimes\O_C\lto B\lto 0.
\end{equation*}
Replacing $H^0B$ by a general subspace of dimension 2 to define a line bundle $M'_{1,B}$ sitting in an exact sequence
\begin{equation*}
0\lto M'_{1,B}\lto M_{1,B}\lto \oplus \O_C\lto 0,
\end{equation*}
we note that $M'_{1,B}=B^{-1}$, hence $M_{1,B}$ and all its exterior powers have 
a filtration whose components are isomorphic to either $M'_{1,B}$ or $\O_C$. The 
vanishing now follows from the assumptions of $H^1(B^{-1}\otimes L)=0=H^1L$.

For the induction step we need to consider several cases:

{\bf Case 1.} $\rank(M_B)=m$: 
The exact sequence on $C^{p+2}$
\begin{equation*}
0\lto \O\big(-\sum_{i=1}^{p+1}\Delta_{i,p+2}\big) \lto\O\big(-\sum_{i=1}^p\Delta_{i,p+2}\big)\lto
\O\big(-\sum_{i=1}^p\Delta_{i,p+2}\big)\otimes\O_{\Delta_{p+1,p+2}}\lto 0
\end{equation*}
yields, after tensoring with $pr_{p+2}^*(B)$ and applying $\pi_{p+2,*}$, an exact sequence on $C^{p+1}$
\begin{equation}\label{eq8}
0\lto M_{p+1,B} \lto \pi_{p+1}^*(M_{p,B})\lto pr_{p+1}^*(B) \otimes\O\big(-\sum_{i=1}^{p}\Delta_{i,p+1}\big) \lto 0.
\end{equation}

We conclude inductively that
\begin{align*}
\det M_{p+1,B} & =\det (\pi_{p+1}^*(M_{p,B}))\otimes pr_{p+1}^*(B^{-1})\otimes\O\big(\sum_{i=1}^p\Delta_{i,p+1}\big) \\
& =pr_1^*(B^{-1})\otimes\dotsm\otimes pr_{p+1}^*(B^{-1})\otimes\O\big(\sum_{1\le i<j\le p+1}\Delta_{i,j}\big),
\end{align*}
and find
\begin{multline*}
H^k\big(C^{p+1},\det(M_B) \otimes pr_1^*(L)\otimes\dotsm\otimes pr_{p+1}^*(L)
\otimes\O\big(-\sum_{1\le i<j\le p+1}\Delta_{i,j}\big)\big) \\
=H^k\big(C^{p+1}, pr_1^*(L\otimes B^{-1})\otimes\dotsm\otimes 
pr_{p+1}^*(L\otimes B^{-1})\big)=0
\end{multline*}
by K\"unneth for $k>0$.

{\bf Case 2.} $\rank(M_B)>m$: The exterior powers of the sequence (\ref{eq8}) lead to a resolution of
$\wedge^m M_{p+1,B}$ which looks as follows
\begin{align*}
\dotsc
\lto & \wedge^{m+2}\pi_{p+1}^*(M_{p,B})\otimes pr_{p+1}^*(B^{-2})\otimes \O\big(\sum_{i=1}^{p}
\Delta_{i,p+1}\big)^{\otimes 2} \\
\lto & \wedge^{m+1} \pi_{p+1}^*(M_{p,B})\otimes pr_{p+1}^*(B^{-1})\otimes\O\big(\sum_{i=1}^{p}\Delta_{i,p+1}\big)
\lto\wedge^m M_{p+1,B}\lto 0.
\end{align*}
The desired vanishing (\ref{eq7}) will result from the vanishings of
\begin{equation}\label{eq9}
\begin{aligned}
H^{k+l}\big(C^{p+1}, \wedge^{m+l+1}\pi_{p+1}^*(M_{p,B})\ \otimes\ & pr_1^*(L)\otimes\dotsm\otimes pr_p^*(L)
\otimes pr_{p+1}^*(L\otimes B^{-l-1}) \\
& \otimes\O\big(\sum_{i=1}^p \Delta_{i,p+1}\big)^{\otimes(l+1)}
\otimes\O\big(-\sum_{1\le i<j\le p+1}\Delta_{i,j}\big)\big)
\end{aligned}
\end{equation}
for all $l\ge 0$ (see e.g., \cite[B.1.2 (i)]{Laz}).

{\bf Case 2.1.} $l=0$: The diagonals involving the component $p+1$ disappear.
We need to consider
\begin{equation*}
H^{k}\big(C^{p+1},\wedge^{m+1}\pi_{p+1}^*(M_{p,B})\otimes pr_1^*(L)\otimes\dotsm pr_p^*(L)
\otimes pr_{p+1}^*(L\otimes B^{-1}) \otimes\O\big(-\sum_{1\le i<j\le p}\Delta_{i,j}\big)\big)
\end{equation*}
and the vanishing follows from K\"unneth on $C^{p+1}=C^p\times C$, using $H^1(L\otimes B^{-1})=0$ on the 
last component and the induction assumption on $C^p$ for the tensor product of the remaining factors.

{\bf Case 2.2.} $l>0$:
Diagonals involving the $(p+1)$-st component now appear with a positive multiplicity $l$.
We wish to use the Leray spectral sequence for $pr_{p+1,*}$. We have to consider $H^1\circ R^{k+l-1}pr_{p+1,*}$ and
$H^0\circ R^{k+l}pr_{p+1,*}$.

The fiber of $R^{k+l-1}pr_{p+1,*}$ of the sheaf in (\ref{eq9}) over a fixed point $x\in C$ consists of
\begin{equation}\label{eq10}
H^{k+l-1}\big(C^p,\wedge^{m+l+1} M_{p,B}\otimes pr_1^*(L+lx)\otimes\dotsm pr_p^*(L+lx)
\otimes\O\big(-\sum_{1\le i<j\le p}\Delta_{i,j}\big)\big).
\end{equation}
As the vanishings of $H^1L$ and $H^1(B^{-1}\otimes L)$ imply the same 
for $L+lx$ in place of $L$, the cohomology group (\ref{eq10}) vanishes by the induction assumption
(note that $k+l-1\ge 1$).

Analogously, the same can be done for the fiber of $R^{k+l}pr_{p+1,*}$ of the sheaf in (\ref{eq9}).
\end{proof}

If the degree of $B$ is large enough, then we can employ a suitable filtration of $M_B$
and obtain a stronger vanishing result. Note that any line bundle of degree
$\ge 2g+p$ is $p$-very ample.

\begin{proposition}\label{prop32}
If $\deg B\ge 2g+2p+1$ and $\deg L\ge 2g+2p$\textup{,} then the vanishing \textup{(\ref{eq6})} holds\textup{,} 
i.e.\textup{,}
\begin{equation*}
H^1\big(C^{p+1}, M_B\otimes pr_1^*(L)\otimes\dotsm \otimes pr_{p+1}^*(L)
\otimes\O(-\sum_{1\le i<j\le p+1}\Delta_{i,j})\big)=0.
\end{equation*}
\end{proposition}
\begin{proof}
{\bf Step 1.}
The proof of \cite[I\!V 6.1]{AG} can be adapted to show that the general line bundle $D$ of 
degree $g+2p+1$ is non-special and $p$-very ample. Here ``general'' means ``after 
excluding a finite number of lower-dimensional subschemes''
of the variety of line bundles of this degree.

{\bf Step 2.}
As $\deg( B\otimes D^{-1})\ge (2g+2p+1)-(g+2p+1)=g$, there is a non-zero homomorphism 
$D\lto B$ vanishing in a finite set of points $z_i$ (possibly with higher multiplicities).

If all points $z_i$ have multiplicity $1$, we obtain an exact sequence
\begin{equation}\label{eq10a}
0\lto M_{p+1,D} \lto M_{p+1,B} \lto \oplus_i \O(-z_i,\dotsc,-z_i)\lto 0
\end{equation}
and the required vanishing will follow from
\begin{equation}\label{eq11}
H^1\big(C^{p+1}, M_{p+1,D}\otimes pr_1^*L\otimes\dotsm\otimes pr_{p+1}^*L\otimes\O\big(-\sum_{1\le i<j\le p+1}\Delta_{i,j}\big)\big)=0.
\end{equation}
and
\begin{equation}\label{eq12}
H^1 \big(C^{p+1},L(-z_i)\otimes\dotsm L(-z_i)\otimes\O\big(-\sum_{1\le i<j\le p+1}\Delta_{i,j}\big)\big)=0
\end{equation}

If there are points with higher multiplicity, then the sheaf on the right-hand side of (\ref{eq10a}) may not split 
into a direct sum, but will have a filtration whose graded pieces have the form (\ref{eq12}).

{\bf Step 3.} As $\deg (L\otimes D^{-1})\ge 2g+2p-(g+2p+1)=g-1$, we have $H^1(L\otimes D^{-1})=0$ for
general $D$, and the vanishing (\ref{eq11}) follows from (\ref{thm31}).

The vanishing (\ref{eq12}) follows by induction (use (\ref{prop32}) for $p-1$) after 
applying $\pi_{p+1,*}$, given that $\deg L(-z_i)\ge 2g+2p-1= 2g+2(p-1)+1$.
\end{proof}

A result for even lower degree of $B$ is possible, and requires $\deg L$ to be higher by $1$.

\begin{proposition}\label{prop36}
If $\deg B\ge 2g+p+1$ and $\deg L\ge 2g+2p+1$\textup{,} then the vanishing 
\textup{(\ref{eq6})} holds\textup{,} i.e.\textup{,}
\begin{equation*}
H^1\big(C^{p+1},M_B\otimes(\otimes_{i=1}^{p+1}(pr_i^*(L))\otimes\O(-\sum_{1\le i<j\le p+1}\Delta_{i,j})\big)=0
\end{equation*}
\end{proposition}
\begin{proof}[Sketch of proof]
{\bf Step 1.}
Successive applications of Leray's spectral sequence show that 
\begin{align*}
H^1 \big(C^{p+1},\ & M_{p+1,B} \otimes pr_1^*(L)\otimes \dotsm\otimes pr_{p+1}^*(L)\otimes 
\O\big(-\sum_{1\le i<j\le p+1}\Delta_{i,j}\big)\big) \\
\cong \enskip & H^1 \big(C^{p+2},pr_1^*(B)\otimes pr_2^*(L)\otimes \dotsm\otimes pr_{p+2}^*(L)\otimes
\O\big(-\sum_{1\le i<j\le p+2}\Delta_{i,j}\big)\big) \\
\cong \enskip & H^1 \big(C^{p+1},M_{p+1,L} \otimes pr_1^*(B)\otimes\dotsm\otimes pr_{p+1}^*(L)\otimes
\O\big(-\sum_{1\le i<j\le p+1}\Delta_{i,j}\big)\big)
\end{align*}
By filtering $M_{p+1,L}$ instead of $M_{p+1,B}$ as in the proof of (\ref{prop32}), we can 
reduce the question to the vanishing of
\begin{equation}\label{eq12a}
H^1 \big(C^{p+1},M_{p+1,D} \otimes pr_1^*(B)\otimes pr_2^*(L)\otimes\dotsm\otimes pr_{p+1}^*(L)\otimes
\O\big(-\sum_{1\le i<j\le p+1}\Delta_{i,j}\big)\big)
\end{equation}
where $D$ is a general line bundle of degree $g+2p+1$.

{\bf Step 2.}
The same steps as in the proof of (\ref{thm31}) can be followed to show that
\begin{equation*}
H^k \big(C^{p+1},\wedge^m M_{p+1,D} \otimes pr_1^*(B)\otimes pr_2^*(L)\otimes\dotsm\otimes pr_{p+1}^*(L)\otimes
\O\big(-\sum_{1\le i<j\le p+1}\Delta_{i,j}\big)\big)=0
\end{equation*}
for all $k\ge 2$, $m>0$, as long as $H^1L=0=H^1(L\otimes D^{-1})$, no assumption on $B$.

{\bf Step 3.} There exist line bundles $D_0,\dotsc,D_p=D$ on $C$ with the following properties:
\begin{itemize}
\item[(i)]
$D_i$ is non-special of degree $g+1+2i$ and $i$-very ample
\item[(ii)]
$H^0(C,D_{i+1}\otimes D_{i}^{-1})$ contains a nonzero section vanishing in two points $x_i, y_i$.
\end{itemize}

{\bf Step 4.}
The vanishing of (\ref{eq12a}) now follows from the case $m=1$ of the next claim:

If $D_p$ is general of degree $\deg(D_p)=g+2p+1$,  
$H^1L=0=H^1(L\otimes D_p^{-1})$ and 
\begin{equation*}
\deg(B)\ge \deg(D_p)-p+m+g-1,
\end{equation*}
then we have
\begin{equation*}
H^1 \big(C^{p+1},\wedge^m M_{p+1,D_p} \otimes pr_1^*(B)\otimes pr_2^*(L)\otimes\dotsm\otimes pr_{p+1}^*(L)\otimes
\O\big(-\sum_{1\le i<j\le p+1}\Delta_{i,j}\big)\big)=0.
\end{equation*}

The induction uses a resolution step  (as in \ref{thm31}), followed by a filtration step
(as in \ref{prop32}). We leave the details to the reader.
\end{proof}

\begin{remark}
Both \textup{(\ref{prop32})} and \textup{(\ref{prop36})} would follow from a more general vanishing result for
\begin{equation*}
H^1 \big(C^n,pr_1^*(L_1)\otimes\dotsm\otimes pr_n^*(L_n)\otimes\O\big(-\sum_{1\le i<j\le n}\Delta_{i,j}\big)\big).
\end{equation*}
It seems likely that it is sufficient to assume that $\deg L_i\ge 2g-3+n+i$; under this condition, 
the vanishing is known for $n\le 3$.
\end{remark}

\section{Additional vanishing results for the canonical bundle}

The filtration approach of \ref{prop32} depends on finding a $p$-very ample line bundle $D$
and a nonzero homomorphism $D\lto B$ with the property that $h^0D=h^0B-(\deg B-\deg D)$
(equivalent to $h^1D=h^1B$). In
geometric terms, the embedding of $C$ by $H^0D$ is the result of a sequence of inner projections
of the embedding by $H^0B$ from the points of a nonzero divisor in $H^0(B\otimes D^{-1})$.

If no such $D$ exists, we obtain information on the geometry of the embedding by $H^0B$.
In the case $B=K$, this information can be interpreted in terms of the varieties of special divisors on $C$.

\begin{proposition}\label{prop38}
Suppose $B$ is $p$-very ample\textup{,} and there exists a $p$-very ample line bundle 
$D$ of degree $\deg(B)-k$\textup{,} and a non-zero homomorphism $D\lto B$ which is 
an isomorphism on $H^1$\textup. If $\deg L\ge 2g+2p$ and 
$\deg L\ge 2g-1+\deg B-2k$\textup{,} then the vanishing \textup{(\ref{eq6})} holds\textup{,}
i.e.\textup{,}
\begin{equation*}
H^1\big(C^{p+1}, M_B\otimes pr_1^*(L)\otimes\dotsm \otimes pr_{p+1}^*(L)
\otimes\O(-\sum_{1\le i<j\le p+1}\Delta_{i,j})\big)=0.
\end{equation*}
\end{proposition}
\begin{proof}
Let $x$ be a point on $C$. The two conditions ``$B(-x)$ is $p$-very ample'' and ``$B(-x)\to B$ 
is an isomorphism on $H^1$'' are both open in $C$. Accordingly our assumptions still hold, if 
we replace $D$ by $B\otimes\O(-E)$ for a general effective divisor $E$ of degree $\deg(B)-\deg(D)$. 

Going back to the proof of proposition (\ref{prop32}), all that needs to be shown is that there exists an effective 
divisor $E$ such that $H^1\big(L\otimes (B\otimes\O(-E))^{-1}\big)=0$. 

Now the variety of such effective divisors $E$ has dimension $k$, whereas the variety of line bundles 
of degree $\deg (L\otimes B^{-1}\otimes\O(E))$ with nontrivial $H^1$ has dimension 
$2g-2-\big(\deg(L)-\deg(B)+k\big)$.

Therefore we are assured of a line bundle with vanishing $H^1$ as long as
$k>2g-2-(\deg(L)-\deg(B)+k)$, i.e., $\deg L>2g-2+\deg(B)-2k$.
\end{proof}

\begin{corollary}\label{cor39}
Suppose there exists a $p$-very ample special line bundle $D$ of degree $g+2p$ with $\dim H^1D=1$. 
Then $K_{p,1}(C,K;L)=0$ for any line bundle $L$ with $\deg L\ge 2g+4p+1$.
\end{corollary}
\begin{proof}
In order to apply \ref{prop38}, we need
$\deg L\ge 2g-1+\deg K-2(\deg K-\deg D)=2g+4p+1$. 
The homomorphism $D\to K$ is due to $\Hom(D,K)\cong H^0(K\otimes D^{-1})\cong H^1(D)^\vee$.
\end{proof}

\begin{proposition}\label{prop310}
Let $C$ be a curve of genus $g$\textup{,} $p\ge 1$\textup{,} $1\le k\le g-2-2p$.
The following conditions are equivalent\textup{:}
\begin{itemize}
\item[\textup{(i)}] There exists a $p$-very ample line bundle $D$ of degree $2g-2-k$ with $\dim H^1D=1$.
\item[\textup{(ii)}] $\dim W^1_{p+j+2} <j$ for all $0\le j\le k-1$.
\end{itemize}
\end{proposition}
\begin{proof}
Either condition implies that $K$ is $p$-very ample. Note also that either condition
for some value of $k$ implies the same condition for all lower values.

(i) implies (ii): 
By induction we can suppose that (ii) holds for all $0\le j\le k-2$. Now suppose 
$\dim W^1_{p+k+1}\ge k-1$, and let $x_1,\dotsc,x_k$ be an arbitrary sequence 
of points (possibly with repetitions). Then the internal projection of $C$ from the 
points $x_1,\dotsc,x_k$ into $\P^{g-k}$ is not $p$-very ample: 
One sees inductively (for each $l\le k-1$) that the subvariety of $W^1_{p+k+1}$ that 
consists of all pencils with $x_1,\dotsc,x_l$ as base locus, has dimension $\ge k-1-l$. 
Considering $l=k-1$, there will be a $g^1_{p+k+1}$ with base locus $x_1,\dotsc,x_{k-1}$, 
hence the projection from $x_1,\dotsc,x_k$ will have a $(p+1)$-secant
$(p-1)$-plane and $K_C\otimes\O_C(-x_1\dotsc-x_k)$ will not be $p$-very ample.

(ii) implies (i): 
Suppose $\dim W^1_{p+j+2}<j$ for all $0\le j\le k-1$, and assume we have determined points $x_1,\dotsc,x_l\in C$ 
so that each of the subvarieties of $W^1_{p+j+2}$ consisting of pencils with base locus $x_1,\dotsc,x_l$ has
dimension $<j-l$. Given an irreducible component of such a subvariety of $W^1_{p+j+2}$, the general point of $C$ will
not be a base point of each of the pencils in this component; i.e., for general $x$, the subvariety of
pencils with base locus $x_1,\dotsc,x_l,x$ will have dimension $<j-l-1$. As each such subvariety of $W^1_{p+j+2}$ has
only finitely many components, and there are only finitely many values of $j$ to consider, we conclude
that there is a point $x_{l+1}\in C$ so that each of the subvarieties of $W^1_{p+j+2}$ consisting of 
pencils with base locus $x_1,\dotsc,x_{l+1}$ has dimension $<j-l-1$. 

Finally, for $l=k$ we have points $x_1,\dotsc,x_k$ with the property that there is no pencil in $W^1_{p+k+2}$
with these points as base locus, i.e., $K_C\otimes\O(-x_1\dotsc -x_k)$ is $p$-very ample.
\end{proof}

\begin{remark}
We list some of the consequences:

1. Suppose $\gon(C)\ge p+3$, i.e., $W^1_{p+2}=\emptyset$ and $K_C$ is $(p+1)$-very ample. 
Then $K_{p,1}(C,K;L)=0$ for $\deg L\ge 4g-5$.
Similarly, if $\gon(C)\ge p+3$ and $\dim W^1_{p+3}\le 0$ (i.e., $W^1_{p+3}$ is empty or zero-dimensional),
then $K_{p,1}(C,K;L)$ vanishes for $\deg L\ge 4g-7$.

2. Suppose $C$ fulfils the conditions of \textup{(\ref{prop310}(ii))} for $k=g-2-2p$, and
let $0\le j\le p$. Then $K_{j,1}(C,K;L)=0$ for $\deg(L)\ge 2g+4j+1$.
\end{remark}

\bigskip
\textsc{Flurstra\ss e 49, 82110 Germering, Germany}

\medskip
\textit{E-mail address}: {\tt Juergen\_Rathmann@yahoo.com}

\end{document}